\documentclass{amsart}
\author{Dilip Raghavan}
\address{Department of Mathematics \\
National University of Singapore\\
Singapore 119076}
\email{raghavan@math.nus.edu.sg}
\urladdr{http://www.math.toronto.edu/raghavan}
\date{\today}
\title{The generic ultrafilter added by ${\left(\FIN \times \FIN\right)}^{+}$}
\usepackage{amssymb, amsmath, amsthm, mathrsfs, enumerate, amsfonts, latexsym, bbm}
\def\polhk#1{\setbox0=\hbox{#1}{\ooalign{\hidewidth
    \lower1.5ex\hbox{`}\hidewidth\crcr\unhbox0}}}
\newtheorem{Theorem}{Theorem}

\newtheorem{Lemma}[Theorem]{Lemma}

\theoremstyle{definition}
\newtheorem{Def}[Theorem]{Definition}

\theoremstyle{remark}

\newcommand{\forces}{\Vdash}
\newcommand{\restrict}{\upharpoonright}

\renewcommand{\c}{\mathfrak{c}}

\newcommand{\h}{\mathfrak{h}}

\renewcommand{\[}{\left[}
\renewcommand{\]}{\right]}
\renewcommand{\P}{\mathbb{P}}

\newcommand{\Q}{\mathbb{Q}}
\newcommand{\lc}{\left|}
\newcommand{\rc}{\right|}

\newcommand\FIN{\mathrm{FIN}}

\DeclareMathOperator{\dom}{dom}

\newcommand{\Pset}{\mathcal{P}}

\newcommand{\BB}{\mathcal{B}}
\newcommand{\A}{{\mathscr{A}}}

\newcommand{\E}{{\mathcal{E}}}
\newcommand{\U}{{\mathcal{U}}}
\newcommand{\XX}{{\mathcal{X}}}
\newcommand{\VV}{{\mathcal{V}}}
\newcommand{\cube}{{\[\omega\]}^{\omega}}

\newcommand{\I}{{\mathcal{I}}}

\newcommand{\V}{{\mathbf{V}}}

\newcommand{\VP}{{\mathbf{V}}^{\P}}

\begin{document}
\begin{abstract}
We investigate the Tukey type of the generic ultrafilter added by the quotient $\Pset(\omega \times \omega) / \left( \FIN \times \FIN \right)$.
We prove that this ultrafilter is not basically generated and yet does not have the maximal Tukey type among direct partial orders of size continuum.
Moreover, any Tukey reduction from this ultrafilter to any other ultrafilter is witnessed by a Baire class one map.
\end{abstract}
\maketitle
\section{Introduction} \label{sec:intro}
The purpose of this short note is to analyze the Tukey type of the generic ultrafilter added by $\Pset(\omega \times \omega) / \left( \FIN \times \FIN \right)$.
Tukey types of ultrafilters (on $\omega$) in general were studied in \cite{dt} and \cite{tukey}.
A particular notion that was analyzed in both these papers is the notion of a basically generated ultrafilter.
This is a property of ultrafilters which guarantees that they are not of the maximal Tukey type (in fact it guarantees that $\langle {\[{\omega}_{1}\]}^{< \omega}, \subset \rangle$ is not Tukey below the ultrafilter).
Moreover it was proved in \cite{tukey} that any monotone map on a basically generated ultrafilter has a nice canonical form that allows its essential features to be captured by a Baire class one map.

In this note we show that the generic ultrafilter added by $\Pset(\omega \times \omega) / \left( \FIN \times \FIN \right)$ is not basically generated and yet does not have maximal Tukey type (Theorems \ref{thm:notbasicallygenerated} and \ref{thm:nottop}).
This is the first known example of such an ultrafilter.
Theorems \ref{thm:nicemaps} and \ref{thm:canonical} provide an exact analogue for this generic ultrafilter of Theorem 17 from \cite{tukey}.
They show that any monotone map on the generic ultrafilter has a nice canonical form.
In particular, there are only $\c$ many ultrafilters Tukey below the generic one.
A noteworthy feature of our results is that they do not require any hypothesis on the ground model.   
\section{Notation} \label{sec:notation}
Let $\I$ denote $\FIN \times \FIN$.
Let $\P$ be ${\I}^{+}$.
Then $\P$ is countably closed and adds a generic ultrafilter $\mathring{\U}$.
If $\U$ is $(\V, \P)$-generic, then in $\V\[\U\]$, $\lc {\c}^{\V} \rc = \lc {\c}^{\V\[\U\]} \rc$.
In particular, $\left\langle {\[{\c}^{\V}\]}^{< \omega}, \subset \right\rangle \; {\equiv}_{T} \; \left\langle {\[{\c}^{\V\[\U\]}\]}^{< \omega}, \subset  \right\rangle$.

For $p \subset \omega \times \omega$ and $n \in \omega$ $p(n) = \{m \in \omega: \langle n, m \rangle \in p\}$.
For $x \in {}^{\omega}{\left( \Pset(\omega) \right)}$ and $a \subset \omega$, $x \restrict a = \{\langle n, m \rangle \in \omega \times \omega: n \in a \ \text{and} \ m \in x(n)\}$.
We will sometimes abuse notation and write $x$ for $x \restrict \omega$.
For $p \in \P$, put $\dom(p) = \{n \in \omega: p(n) \neq 0 \}$.
Let us say that $p \in {\I}^{+}$ is \emph{standard} if $\forall n \in \dom(p)\[\lc p(n) \rc = \omega\]$.
It is clear that for every $p \in \P$ there is $q \subset p$ which is standard.  
\begin{Def} \label{def:cofinalandbase}
	Let $\U$ be an ultrafilter on $\omega$. A set $\BB \subset \U$ is said to be a \emph{filter base} for $\U$ if $\forall a \in \U \exists b \in \BB\[b \subset a\]$ and if $\forall {b}_{0}, {b}_{1} \in \BB\[\; {b}_{0} \cap {b}_{1} \in \BB\]$.  
\end{Def}
\begin{Def} \label{def:basicallygenerated}
	Let $\U$ be an ultrafilter on $\omega$. We say that $\U$ is \emph{basically generated} if there is a filter base $\BB \subset \U$ with the property that for every $\langle {b}_{n}: n \in \omega \rangle \subset \BB$ and $b \in \BB$, if $\langle {b}_{n}: n \in \omega \rangle$ converges to $b$ (with respect to the usual topology on $\Pset(\omega)$), then there exists $X \in \cube$ such that ${\bigcap}_{n \in X}{{b}_{n}} \in \U$. 
\end{Def}
We will be dealing with ultrafilters on $\omega \times \omega$.
Definitions \ref{def:cofinalandbase} and \ref{def:basicallygenerated} apply to such ultrafilters too with the obvious modifications.

Let $a \in \cube$.
If $\A \subset \Pset(a)$, $\I(\A)$ denotes the ideal on $a$ generated by $\A$ together with the Frechet filter.
\section{The Results} \label{sec:results}
We first give a direct argument that $\mathring{\U}$ is not of the maximal Tukey type.
After that we show that monotone maps defined on $\mathring{\U}$ can always be ``captured'' by Baire class one maps.
\begin{Theorem} \label{thm:nottop}
Let $\U$ be $(\V, \P)$-generic.
Then in $\V\[\U\]$ $\left\langle {\[{\c}^{\V\[\U\]}\]}^{< \omega}, \subset  \right\rangle \; {\not\leq}_{T} \; \U$.
\begin{proof}
Suppose not.
Working in $\V$ find $\{{\mathring{x}}_{\alpha}: \alpha < {\c}^{\V}\} \subset {\V}^{\P}$ such that for each $\alpha < {\c}^{\V}$, $\forces {{\mathring{x}}_{\alpha} \in \mathring{\U}}$, and also a standard ${p}_{0} \in \P$ such that for any $X \in {\[{\c}^{\V}\]}^{\omega}$, ${p}_{0} \forces {\bigcap}_{\alpha \in X}{{\mathring{x}}_{\alpha} \notin \mathring{\U}}$.
Let $\langle {p}_{\alpha}: \alpha < {\c}^{\V} \rangle$ be an enumeration of $\{p \in \P: p \leq {p}_{0}\}$ such that each element of it occurs cofinally often.
For each $\alpha < {\c}^{\V}$ choose a standard ${q}_{\alpha} \in \P$ and ${x}_{\alpha} \in \P$ such that ${q}_{\alpha} \subset {p}_{\alpha} \cap {p}_{0} \cap {x}_{\alpha}$ and ${q}_{\alpha} \; \forces \; {\mathring{x}}_{\alpha} = {x}_{\alpha}$.
Suppose for a moment that we can find $X \in {\[{\c}^{\V}\]}^{\omega}$ such that ${\bigcap}_{\alpha \in X}{{q}_{\alpha}} \in \P$.
Then putting $q = {\bigcap}_{\alpha \in X}{{q}_{\alpha}}$, it is clear that $q \forces {\bigcap}_{\alpha \in X}{{\mathring{x}}_{\alpha}} \in \mathring{\U}$, which is a contradiction as $q \leq {p}_{0}$.

To find such $X$, let $\E$ be $(\V, \Pset(\omega) / \FIN)$-generic with $\dom({p}_{0}) \in \E$.
In $\V\[\E\]$ consider the poset $\Q = {\left( \Pset(\omega) / \FIN \right)}^{\omega}$.
Define ${x}_{0} \in \Q$ as follows.
For any $n \in \dom({p}_{0})$, ${x}_{0}(n) = {p}_{0}(n)$.
For any $n \notin \dom({p}_{0})$, ${x}_{0}(n) = \omega$.
Let $G$ be $(\V\[\E\], \Q)$-generic with ${x}_{0} \in G$.
In $\V[\E][G]$ define for each $n \in \omega$, ${\VV}_{n} = \{x(n): x \in G\}$.
It is clear that $\E$ and each ${\VV}_{n}$ are selective ultrafilters in $\V\[\E\]\[G\]$.
Put $\VV = {\bigotimes}_{\E}{{\VV}_{n}}$.
We claim that for each $\alpha < {\c}^{\V}$, there is ${\c}^{\V} > \beta \geq \alpha$ such that ${q}_{\beta} \in \VV$.
Note that this is sufficient to find $X \in \V \cap {\[{\c}^{\V}\]}^{\omega}$ such that ${\bigcap}_{\alpha \in X}{{q}_{\alpha}} \in \P$.
This is because in $\V\[\E\]\[G\]$ there will be $X \in {\[{\c}^{\V}\]}^{\omega}$ such that ${\bigcap}_{\alpha \in X}{{q}_{\alpha}} \in \VV \subset \P$ (this is because $\VV$ is basically generated; see Lemma \ref{lem:bv} below).
And since no new countable sets of ordinals were added $X \in \V \cap {\[{\c}^{\V}\]}^{\omega}$.

In order to prove the claim, fix $\alpha < {\c}^{\V}$.
Working in $\V\[\E\]$ define 
\begin{align*}
	{D}_{\alpha} = \{y \in \Q: \exists \beta \geq \alpha \exists a \in \E \[\dom({q}_{\beta}) = a \ \text{and} \ {q}_{\beta} = y \restrict a\]\}.
\end{align*}  
Let us check that ${D}_{\alpha}$ is dense below ${x}_{0}$.
Fix $x \leq {x}_{0}$.
Note that $x \in \V$.
Working in $\V$, define $D(\alpha, x) = \{\dom({q}_{\beta}): \beta \geq \alpha \ \text{and} \ {q}_{\beta} \subset x \restrict \omega\}$.
We claim that $D(\alpha, x)$ is dense below $\dom({p}_{0})$ in $\Pset(\omega) / \FIN$.
Fix $a \in {\[\dom({p}_{0})\]}^{\omega}$.
Note that $x \restrict a \in \P$ and that $x \restrict a \leq {p}_{0}$.
So there exists $\beta \geq \alpha$ such that ${q}_{\beta} \subset x \restrict a \subset x \restrict \omega$.
It is clear that $\dom({q}_{\beta}) \subset a$ and is as needed.
Now, back in $\V\[\E\]$, this means that there is some $a \in \E$ and $\beta \geq \alpha$ such that $a = \dom({q}_{\beta})$ and ${q}_{\beta} \subset x \restrict \omega$.
Define $y \in \Q$ as follows.
If $n \in a$, then $y(n) = {q}_{\beta}(n)$ and if $n \notin a$, then $y(n) = x(n)$.
It is clear that $y \leq x$ and is as needed.
Therefore, in $\V\[\E\]\[G\]$, there exists $y \in G$, $\beta \geq \alpha$ and $a \in \E$ such that $\dom({q}_{\beta}) = a$ and $y \restrict a = {q}_{\beta}$.
But this means that ${q}_{\beta} \in \VV$ and we are done.
\end{proof}
\end{Theorem}
We remark that our argument does not show that in $\V\[\U\]$, $\langle {\[{\omega}_{1}\]}^{< \omega}, \subset \rangle \; {\not\leq}_{T} \; \U$.
However, it is easy to see that if in $\V$, $\h(\Pset(\omega \times \omega) / \left( \FIN \times \FIN \right)) > {\omega}_{1}$, then in $\V\[\U\]$, $\langle {\[{\omega}_{1}\]}^{< \omega}, \subset \rangle \; {\not\leq}_{T} \; \U$ holds.
We do not know if it is possible to prove this in general.

Let $\XX \subset \Pset(\omega)$.
Recall that a map $\phi:\XX \rightarrow \Pset(\omega)$ is said to be \emph{monotone} if $\forall a, b \in \XX \[b \subset a \implies \phi(b) \subset \phi(a)\]$.
Such a map is said to be \emph{non-zero} if $\forall a \in \XX\[\phi(a) \neq 0 \]$.

Next we will show that any monotone maps defined on $\mathring{\U}$ has a ``nice'' canonical form similar to what is obtained in Section 4 of \cite{tukey}.
This will imply that if $\U$ is $(\V, \P)$-generic, then in $\V\[\U\]$ there are only $\c$ many ultrafilters that are Tukey below $\U$.
This gives another, albeit less direct, proof that $\U$ is not of the maximal cofinal type for directed sets of size continuum.
The proof will go through the corresponding result for Fubini products of selective ultrafilters.
Recall the following definitions and results which appear in \cite{tukey}.
\begin{Def}\label{def:psi}
	Let $\XX \subset \Pset(\omega)$ and let $\phi: \XX \rightarrow \Pset(\omega)$. Define ${\psi}_{\phi}: \Pset(\omega) \rightarrow \Pset(\omega)$ by ${\psi}_{\phi}(a) = \{k \in \omega: \forall b \in \XX\[a \subset b \implies k \in \phi(b)\] \} = \bigcap\{\phi(b): b \in \XX \wedge a \subset b\}$, for each $a \in \Pset(\omega)$.  
\end{Def}
\begin{Lemma} [Lemma 16 of \cite{tukey}] \label{lem:findetermines}
	Let $\U$ be basically generated by $\BB \subset \U$. Let $\phi: \BB \rightarrow \Pset(\omega)$ be a monotone map such that $\phi(b) \neq 0$ for every $b \in \BB$. Let $\psi = {\psi}_{\phi}$. Then for every $b \in \BB$, ${\bigcup}_{s \in {\[b\]}^{< \omega}}{\psi(s)} \neq 0$. 
\end{Lemma}
Once again, Definition \ref{def:psi} and Lemma \ref{lem:findetermines} apply to $\Pset(\omega \times \omega)$ with the obvious modifications.

Let $\E$ and $\langle {\VV}_{n}: n \in \omega \rangle$ be selective ultrafilters.
Put $\VV = {\bigotimes}_{\E}{{\VV}_{n}}$.
Consider ${\BB}_{\VV} = \{b \subset \omega \times \omega: \dom(b) \in \E \ \text{and} \ \forall n \in \dom(b)\[b(n) \in {\VV}_{n}\]\}$.
Then the following is easy to prove.
For a more general statement see \cite{tukey}.
\begin{Lemma} \label{lem:bv}
$\VV$ is basically generated by ${\BB}_{\VV}$. 
\end{Lemma}
\begin{Theorem} \label{thm:nicemaps}
Let $\U$ be $(\V, \P)$-generic.
In $\V\[\U\]$, let $\phi: \U \rightarrow \Pset(\omega)$ be a monotone non-zero map.
Then there exist $P \subset {\[\omega \times \omega\]}^{< \omega}$ and $\psi: P \rightarrow \omega$ such that
	\begin{enumerate} 
		\item
			$\forall a \in \U\[P \cap {\[a\]}^{< \omega} \neq 0 \]$.
		\item		
			$\forall a \in \U \exists b \in \U \cap {\[a\]}^{\omega}\forall s \in P \cap {\[b\]}^{< \omega}\[\psi(s) \in \phi(b)\]$.
	\end{enumerate}
\end{Theorem}
\begin{proof}
The proof is similar to the proof of Theorem \ref{thm:nottop}.
Suppose that the theorem fails.
Fix $\mathring{\phi} \in {\V}^{\P}$ such that $\forces {\mathring{\phi}: \mathring{\U} \rightarrow \Pset(\omega) \ \text{is a monotone non-zero map}}$.
Fix a standard ${p}_{0} \in \P$ such that for any $P \subset {\[\omega \times \omega\]}^{< \omega}$ and $\psi: P \rightarrow \omega$, 
\begin{align*}
	{p}_{0} \; \forces \; ``\text{either} & \ \exists a  \in \mathring{\U}\[P \cap {\[a\]}^{< \omega} = 0 \] \\ & \text{or} \ \exists a \in \mathring{\U} \forall b \in \mathring{\U} \cap {\[a\]}^{\omega} \exists s \in P \cap {\[b\]}^{< \omega}\[\psi(s) \notin \mathring{\phi}(b)\]''.
\end{align*}
Let $\{\langle {p}_{\alpha}, {A}_{\alpha}, {\psi}_{\alpha} \rangle: \alpha < {\c}^{\V} \}$ enumerate all triples $\langle p, A, \psi \rangle$ such that $p \in \P$ and $p \leq {p}_{0}$, $A \subset {\[\omega \times \omega\]}^{< \omega}$, and $\psi: A \rightarrow \omega$.
Define $\chi: \P \rightarrow \Pset(\omega)$ by $\chi(p) = \left\{k \in \omega: \exists q \leq p \[q \forces k \in \mathring{\phi}(p)\]\right\}$.
Observe that if $q \leq p$, then $q \forces p \in \mathring{\U}$, and hence $q \forces \mathring{\phi}(p) \ \text{is defined}$.
Next, it is easy to check that $\chi$ is monotone.
Moreover, $p \forces \mathring{\phi}(p) \neq 0$.
Therefore, for some $q \leq p$ and $k \in \omega$, $q \forces k \in \mathring{\phi}(p)$, whence $k \in \chi(p)$.
Thus $\chi$ is monotone and non-zero.
Now build a sequence $\langle {q}_{\alpha}: \alpha < {\c}^{\V} \rangle$ with the following properties:
\begin{enumerate}
	\item[(3)]
		${q}_{\alpha} \in \P$, ${q}_{\alpha}$ is standard, and ${q}_{\alpha} \subset {p}_{\alpha}$.
	\item[(4)]
		either ${A}_{\alpha} \cap {\[{q}_{\alpha}\]}^{< \omega} = 0$ or for some $s \in {A}_{\alpha} \cap {\[{q}_{\alpha}\]}^{< \omega}$, ${\psi}_{\alpha}(s) \notin \chi({q}_{\alpha})$.
\end{enumerate} 
To see how to build such a sequence, fix $\alpha < {\c}^{\V}$.
Let $\U$ be $(\V, \P)$-generic with ${p}_{\alpha} \in \U$.
Since ${p}_{\alpha} \leq {p}_{0}$, in $\V\[\U\]$ either there is $a \in \U$ such that ${A}_{\alpha} \cap {\[a\]}^{< \omega} = 0$ or there is $a \in \U$ such that for all $b \in \U \cap {\[a\]}^{\omega}$, there exists $s \in {A}_{\alpha} \cap {\[b\]}^{< \omega}$ such that ${\psi}_{\alpha}(s) \notin \mathring{\phi}\[\U\](b)$.
Suppose that the first case happens.
Let ${q}_{\alpha}$ be a standard element of $\P$ such that ${q}_{\alpha} \subset {p}_{\alpha} \cap a$.
Then ${\[{q}_{\alpha}\]}^{< \omega} \cap {A}_{\alpha} \subset {\[a\]}^{< \omega} \cap {A}_{\alpha} = 0$.

Now suppose that the second case happens in $\V\[\U\]$.
Working in $\V\[\U\]$ fix $a \in \U$ as in the second case.
Let $b \in \U$ be standard such that $b \subset {p}_{\alpha} \cap a$.
Since $b \in \U \cap {\[a\]}^{\omega}$ there is $s \in {A}_{\alpha} \cap {\[b\]}^{< \omega}$ such that ${\psi}_{\alpha}(s) \notin \mathring{\phi}\[\U\](b)$.
Find ${q}^{\ast} \in \U$ such that (in $\V$) ${q}^{\ast} \forces {\psi}_{\alpha}(s) \notin \mathring{\phi}(b)$.
Let $q \in \U$ be standard so that $q \subset b \cap {q}^{\ast}$.
Back in $\V$, define ${q}_{\alpha}$ as follows.
For $n \in \dom(s)$, put ${q}_{\alpha}(n) = b(n)$.
If $n \in \omega \setminus \dom(s)$, then ${q}_{\alpha}(n) = q(n)$.
Note that ${q}_{\alpha} \in \P$, it is standard, and $s \subset {q}_{\alpha} \subset b \subset {p}_{\alpha}$.
Moreover, if $\langle n, m \rangle \in {q}_{\alpha} \setminus q$, then $n \in \dom(s)$.
As $\dom(s)$ is finite, ${q}_{\alpha} \setminus q \in \I$.
Therefore, ${q}_{\alpha} \leq q$ and ${q}_{\alpha} \forces {\psi}_{\alpha}(s) \notin \mathring{\phi}(b)$.
Note that $s \in {A}_{\alpha} \cap {\[{q}_{\alpha}\]}^{< \omega}$.
To see that ${\psi}_{\alpha}(s) \notin \chi({q}_{\alpha})$, suppose for a contradiction that there is $r \leq {q}_{\alpha}$ such that $r \forces {\psi}_{\alpha}(s) \in \mathring{\phi}({q}_{\alpha})$.
As ${q}_{\alpha} \subset b$, $r \forces {\psi}_{\alpha}(s) \in \mathring{\phi}(b)$, which is impossible.
This completes the construction of ${q}_{\alpha}$.

Just as in the proof of Theorem \ref{thm:nottop}, let $\E$ be $(\V, \Pset(\omega) / \FIN)$-generic with $\dom({p}_{0}) \in \E$.
In $\V\[\E\]$ consider the poset $\Q = {\left( \Pset(\omega) / \FIN \right)}^{\omega}$.
Define ${x}_{0} \in \Q$ as follows.
For any $n \in \dom({p}_{0})$, ${x}_{0}(n) = {p}_{0}(n)$.
For any $n \notin \dom({p}_{0})$, ${x}_{0}(n) = \omega$.
Let $G$ be $(\V\[\E\], \Q)$-generic with ${x}_{0} \in G$.
In $\V[\E][G]$ define for each $n \in \omega$, ${\VV}_{n} = \{x(n): x \in G\}$.
It is clear that $\E$ and each ${\VV}_{n}$ are selective ultrafilters in $\V\[\E\]\[G\]$.
Put $\VV = {\bigotimes}_{\E}{{\VV}_{n}}$.
Then $\VV$ is basically generated by ${\BB}_{\VV}$.
Note that ${\BB}_{\VV} \subset \VV \subset \P$.
Put $\phi = \chi \restrict {\BB}_{\VV}$.
Note that the hypotheses of Lemma \ref{lem:findetermines} are satisfied.
Put $A = \{s \in {\[\omega \times \omega\]}^{< \omega}: {\psi}_{\phi}(s) \neq 0 \}$.
Define $\psi: A \rightarrow \omega$ by $\psi(s) = \min({\psi}_{\phi}(s))$ for any $s \in A$.
We claim that there exists $\alpha < {\c}^{\V}$ such that ${q}_{\alpha} \in {\BB}_{\VV}$ and ${A}_{\alpha} = A$ and ${\psi}_{\alpha} = \psi$.
Suppose for a moment that this claim is true.
Applying Lemma \ref{lem:findetermines} to ${q}_{\alpha}$ find $s \in {\[{q}_{\alpha}\]}^{< \omega}$ such that ${\psi}_{\phi}(s) \neq 0$.
So $s \in {A}_{\alpha} \cap {\[{q}_{\alpha}\]}^{< \omega}$.
Moreover, by the definition of ${\psi}_{\phi}$, for any $t \in {A}_{\alpha} \cap {\[{q}_{\alpha}\]}^{< \omega}$, ${\psi}_{\phi}(t) \subset \phi({q}_{\alpha}) = \chi({q}_{\alpha})$. 
This means that for every $t \in {A}_{\alpha} \cap {\[{q}_{\alpha}\]}^{< \omega}$, ${\psi}_{\alpha}(t) \in \chi({q}_{\alpha})$.
But this contradicts the way ${q}_{\alpha}$ was constructed.

To prove the claim first note that $A$ and $\psi$ are in $\V$.
In $\V\[\E\]$ define $D(A, \psi)$ as
\begin{align*}
	 \{y \in \Q: \exists a \in \E \exists \alpha < {\c}^{\V} \[\dom({q}_{\alpha}) = a, y \restrict a = {q}_{\alpha}, {A}_{\alpha} = A, \ \text{and} \ {\psi}_{\alpha} = \psi \]\}.
\end{align*}
We argue that $D(A, \psi)$ is dense below ${x}_{0}$.
Fix $x \in \Q$ with $x \leq {x}_{0}$.
Note that $x \in \V$.
Working in $\V$ define $D(x, A, \psi) = \{\dom({q}_{\alpha}): \alpha < {c}^{\V}, {q}_{\alpha} \subset x \restrict \omega, {A}_{\alpha} = A, \ \text{and} \ {\psi}_{\alpha} = \psi \}$.
To see that $D(x, A, \psi)$ is dense below $\dom({p}_{0})$ fix $a \in {\[\dom({p}_{0})\]}^{\omega}$. 
Put $p = x \restrict a$ and note that $p \in \P$ and that $p \leq {p}_{0}$.
Therefore, there exists $\alpha < {\c}^{\V}$ such that ${p}_{\alpha} = p$, ${A}_{\alpha} = A$, and ${\psi}_{\alpha} = \psi$.
Thus ${q}_{\alpha} \subset x \restrict a \subset x \restrict \omega$.
Also $\dom({q}_{\alpha}) \subset a$.
Therefore $\dom({q}_{\alpha})$ is as needed.
Back in $\V\[\E\]$, fix $a \in \E$ and $\alpha < {\c}^{\V}$ such that $\dom({q}_{\alpha}) = a$, ${q}_{\alpha} \subset x \restrict \omega$, ${A}_{\alpha} = A$, and ${\psi}_{\alpha} = \psi$.
For $n \in a$, put $y(n) = {q}_{\alpha}(n)$.
For $n \in \omega \setminus a$, put $y(n) = x(n)$.
Then $y \in \Q$ and $y \leq x$.
It is clear that $y \in \Q$ and that $y \leq x$.
Also $y \restrict a = {q}_{\alpha}$ and so it is clear that $y$ is as needed.
So in $\V\[\E\]\[G\]$, there is $y \in G$, $a \in \E$, and $\alpha < {\c}^{\V}$ such that $\dom({q}_{\alpha}) = a$, $y \restrict a = {q}_{\alpha}$, ${A}_{\alpha} = A$, and ${\psi}_{\alpha} = \psi$.
Since $\dom({q}_{\alpha}) = a \in \E$ and for all $n \in \dom({q}_{\alpha})$, ${q}_{\alpha}(n) = y(n) \in {\VV}_{n}$, ${q}_{\alpha} \in {\BB}_{\VV}$, and we are done.
\end{proof}
Now we show that the conclusion of Theorem 17 of \cite{tukey}, which was proved there to hold for all basically generated ultrafilters, also holds for $\mathring{\U}$.
\begin{Def} \label{def:UP}
	Let $\U$ be an ultrafilter on $\omega \times \omega$, and let $P \subset {\[\omega \times \omega\]}^{< \omega} \setminus \{0\}$.
We define ${\U}(P) = \{A \subset P: \exists a \in \U\[P \cap {\[a\]}^{< \omega} \subset A \]\}$.
\end{Def}   
If $\forall a \in \U\[\lc P \cap {\[a \]}^{< \omega} \rc = \omega\]$, then ${\U}(P)$ is a proper, non-principal filter on $P$.
The following theorem says that any Tukey reduction from $\mathring{\U}$ is given by an Rudin-Keisler reduction from $\mathring{\U}(P)$ for some $P$.
\begin{Theorem} \label{thm:canonical}
Let $\U$ be $(\V, \P)$-generic.
In $\V\[\U\]$, let $\VV$ be an arbitrary ultrafilter so that $\VV \; {\leq}_{T} \; \U$.
Then there is $P \subset {\[\omega \times \omega\]}^{< \omega} \setminus \{0\}$ such that
	\begin{enumerate}
		\item
			$\forall t, s \in P \[t \subset s \implies t = s \]$		
		\item
			$\U(P) \; {\equiv}_{T} \; \U$
		\item
			$\VV \; {\leq}_{RK} \; \U(P)$
	\end{enumerate}
\end{Theorem}  
\begin{proof}
The proof is almost the same as the proof of Theorem 17 of \cite{tukey}.
Work in $\V\[\U\]$.
Fix an ultrafilter $\VV$ and a map $\phi: \U \rightarrow \VV$ which is monotone and cofinal in $\VV$.
Since $\phi$ is monotone and non-zero, fix $A \subset {\[\omega \times \omega\]}^{< \omega}$ and $\psi: A \rightarrow \omega$ as in Theorem \ref{thm:nicemaps}.
First we claim that $0 \notin A$.
Indeed suppose for a contradiction that $0 \in A$ and let $k = \psi(0)$.
Let $e \in \VV$ be such that $k \notin e$ and let $a \in \U$ be such that $\phi(a) \subset e$.
By (2) of Theorem \ref{thm:nicemaps} there is $b \in \U \cap {\[a\]}^{\omega}$ such that for all $s \in A \cap {\[b\]}^{< \omega}$, $\psi(s) \in \phi(b)$.
However, $0 \in A \cap {\[b\]}^{< \omega}$, and so $k = \psi(0) \in \phi(b) \subset \phi(a) \subset e$, a contradiction.
Thus $0 \notin A$.
Define
\begin{align*}
	P = \{s \in A: s \ \text{is minimal in} \ A \ \text{with respect to} \ \subset \}. 
\end{align*}
It is clear that $P \subset {\[\omega \times \omega\]}^{< \omega} \setminus \{0\}$ and that $P$ satisfies (1) by definition. 

Next, for any $a \in \U$, ${\bigcup}{\left(P \cap {\[a\]}^{< \omega}\right)} \in \U$.
To see this, fix $a \in \U$, and suppose that $a \setminus \left({\bigcup}{\left(P \cap {\[a\]}^{< \omega}\right)}\right) \in \U$.
By (1) of Theorem \ref{thm:nicemaps}, fix $s \in A$ with $s \subset a \setminus \left({\bigcup}{\left(P \cap {\[a\]}^{< \omega}\right)}\right)$.
However there is $t \in P$ with $t \subset s$, whence $t = 0$, an impossibility.
It follows from this that for each $a \in \U$, $P \cap {\[a\]}^{< \omega}$ is infinite.     

Next, verify that $\U(P) \; {\equiv}_{T} \; \U$.
Define $\chi: \U \rightarrow \U(P)$ by $\chi(a) = P \cap {\[a\]}^{< \omega}$, for each $a \in \U$.
This map is clearly monotone and cofinal in $\U(P)$.
So $\chi$ is a convergent map.
On the other hand, $\chi$ is also Tukey.
To see this, fix $\XX \subset \U$, unbounded in $\U$.
Assume that $\{\chi(a): a \in \XX\}$ is bounded in $\U(P)$.
So there is $b \in \U$ such that $P \cap {\[b\]}^{< \omega} \subset P \cap {\[a\]}^{< \omega}$ for each $a \in \XX$.
However $c = {\bigcup}{\left(P \cap {\[b\]}^{< \omega}\right)} \in \U$.
Now, it is clear that $c \subset a$, for each $a \in \XX$, a contradiction.  

Next, check that $\VV \; {\leq}_{RK} \; \U(P)$.
Define $f: P \rightarrow \omega$ by $f = \psi \restrict P$.
Fix $e \subset \omega$, and suppose first that ${f}^{-1}(e) \in \U(P)$.
Fix $a \in \U$ with $P \cap {\[a\]}^{< \omega} \subset {f}^{-1}(e)$. 
If $e \notin \VV$, then $\omega \setminus e \in \VV$, and there exists $c \in \U$ with $\phi(c) \subset \omega \setminus e$.
By (2) of Theorem \ref{thm:nicemaps} fix $b \in \U \cap {\[a \cap c\]}^{\omega}$ such that for all $s \in A \cap {\[b\]}^{< \omega}$, $\psi(s) \in \phi(b)$.
By (1) of Theorem \ref{thm:nicemaps}, fix $s \in A \cap {\[b\]}^{< \omega}$. 
Fix $t \subset s$ with $t \in P$.
Let $k = f(t) = \psi(t)$.
As $t \subset s \subset b \subset a$, $t \in P \cap {\[a\]}^{< \omega} \subset {f}^{-1}(e)$.
Thus $k \in e$.
On the other hand, since $t \in A \cap {\[b\]}^{< \omega}$, $\psi(t) \in \phi(b)$.
So $k \in \phi(b) \subset \phi(c) \subset \omega \setminus e$, a contradiction.

Next, suppose that $e \in \VV$.
By cofinality of $\phi$, there is $a \in \U$ such that $\phi(a) \subset e$.
Applying (2) of Theorem \ref{thm:nicemaps}, fix $b \in \U \cap {\[a\]}^{\omega}$ such that for all $s \in A \cap {\[b\]}^{< \omega}$, $\psi(s) \in \phi(b)$.
Now, if $s \in P \cap {\[b\]}^{< \omega}$, then $f(s) = \psi(s) \in \phi(b) \subset \phi(a) \subset e$.
Therefore, $P \cap {\[b\]}^{< \omega} \subset {f}^{-1}(e)$, whence ${f}^{-1}(e) \in \U(P)$.  
\end{proof}
An immediate corollary of Theorem \ref{thm:canonical} is that if $\U$ is $(\V, \P)$-generic, then in $\V\[\U\]$, $\{\VV: \VV \ \text{is an ultrafilter on} \ \omega \ \text{and} \ \VV \; {\leq}_{T} \; \U\}$ has size $\c$.

Next we show that $\mathring{\U}$ is not basically generated.
As far as we are aware, this is the first example (even consistently) of an ultrafilter that is not basically generated and whose cofinal type is not maximal.
Thus our result establishes the consistency of the statement ``$\exists \U\[\U \ \text{is not basically generated} \ \text{and} \ {\[\c\]}^{< \omega} \; {\not\leq}_{T} \; \U\]$''.  
\begin{Theorem} \label{thm:notbasicallygenerated}
$\forces {\mathring{\U} \ \text{is not basically generated}}$.
\end{Theorem}
\begin{proof}
	Let $\mathring{\BB} \in \VP$ be such that
\begin{enumerate}
	\item
		$\forces \mathring{\BB} \subset \mathring{\U}$
	\item
		$\forces \forall a \in \mathring{\U} \exists b \in \mathring{\BB}\[b \subset a \]$
\end{enumerate}
Let ${p}^{\ast} \in \P$ be standard such that 
\begin{align*}
	{p}^{\ast} \forces ``&\text{every convergent sequence from} \ \mathring{\BB}  \\ & \text{contains an infinite sub-sequence bounded in} \ \mathring{\U}''
\end{align*}
Now build two sequences $\{{p}_{\alpha}: \alpha < {\omega}_{1}\}$ and $\{{x}_{\alpha}: \alpha < {\omega}_{1}\}$ with the following properties.
\begin{enumerate}
	\item[(3)]
		${p}_{\alpha} \subset {p}^{\ast}$, both ${p}_{\alpha}$ and ${x}_{\alpha}$ are elements of $\P$, ${p}_{\alpha}$ is standard, ${p}_{\alpha} \subset {x}_{\alpha}$, and ${p}_{\alpha} \forces {x}_{\alpha} \in \mathring{\BB}$.
	\item[(4)]
		$\forall \xi < \alpha \[{x}_{\alpha} \leq {p}_{\xi}\]$ (therefore, $\forall \xi < \alpha \[{p}_{\alpha} \leq {x}_{\alpha} \leq {p}_{\xi}\]$).
	\item[(5)]
		$\forall n \in \omega \[\{\xi < \alpha: \lc \left( {x}_{\alpha} \cap {x}_{\xi} \right) (n) \rc = \omega \} \ \text{is finite}\]$.
	\item[(6)]
		for each $\alpha < {\omega}_{1}$ and $n \in \dom({p}_{\alpha})$ let $F(\alpha, n) = \{\xi \leq \alpha: {p}_{\alpha}(n) \; {\subset}^{\ast} \; {x}_{\xi}(n)\}$.
Note that $\alpha \in F(\alpha, n)$.
Let $G(\alpha, n) = \{{p}_{\alpha}(n) \cap {x}_{\xi}(n): \xi \in {\omega}_{1} \setminus F(\alpha, n)\}$.
Then $\I(G(\alpha, n))$ is a proper ideal on ${p}_{\alpha}(n)$ for each $\alpha < {\omega}_{1}$ and $n \in \dom({p}_{\alpha})$. 
\end{enumerate}
Suppose for a moment that such sequences can be constructed.
Let $\delta < {\omega}_{1}$ and $\{{\alpha}_{0} < {\alpha}_{1} < \dotsb \} \subset \delta$ be such that $\{\langle {x}_{{\alpha}_{i}}, {p}_{{\alpha}_{i}} \rangle: i \in \omega\}$ converges to $\langle {x}_{\delta}, {p}_{\delta} \rangle$.
Note that for each $i \in \omega$, ${p}_{\delta} \leq {p}_{{\alpha}_{i}}$.
Therefore ${p}_{\delta} \forces \{{x}_{{\alpha}_{i}}: i < \omega\} \cup \{{x}_{\delta}\} \subset \mathring{\BB}$. 
Since ${p}_{\delta} \leq {p}^{\ast}$ and since $\P$ does not add any countable sets of ordinals, there exist $X \in \cube$ and a standard $q \in \P$ such that $q \subset {x}_{\delta}$ and $\forall i \in X \[q \subset {x}_{{\alpha}_{i}}\]$.
Fix $n \in \dom(q)$.
Then for each $i \in X$, $q(n) \subset {x}_{\delta}(n) \cap {x}_{{\alpha}_{i}}(n)$.
So $\{{\alpha}_{i}: i \in X\} \subset \{\alpha < \delta: \lc \left( {x}_{\delta} \cap {x}_{\alpha} \right) (n) \rc = \omega \}$, contradicting (5).

To see how to build such sequences, note first that if $\delta \leq {\omega}_{1}$ is a limit ordinal and if for each $\beta < \delta$ the sequences $\langle {x}_{\alpha}: \alpha < \beta \rangle$ and $\langle {p}_{\alpha}: \alpha < \beta \rangle$ do not contain any witnesses violating clauses (3)-(6), then the sequences $\langle {x}_{\alpha}: \alpha < \delta \rangle$ and $\langle {p}_{\alpha}: \alpha < \delta \rangle$ do not contain any such witnesses either.
Therefore, fix $\alpha < {\omega}_{1}$ and assume that $\langle {x}_{\xi}: \xi < \alpha \rangle$ and $\langle {p}_{\xi}: \xi < \alpha \rangle$ are given to us.
We only need to worry about finding ${x}_{\alpha}$ and ${p}_{\alpha}$.
First if $\alpha = 0$, then fix a $(\V, \P)$-generic $\U$ with ${p}^{\ast} \in \U$.
In $\V\[\U\]$ fix ${x}_{0} \in \mathring{\BB}\[\U\]$ with ${x}_{0} \subset {p}^{\ast}$.
In $\V$, fix a standard ${p}_{0} \in \P$ such that ${p}_{0} \subset {x}_{0}$ and ${p}_{0} \forces {x}_{0} \in \mathring{\BB}$.
It is clear that (3) is satisfied, and (4)-(6) are trivially true.
So assume $\alpha > 0$.
Let $\{{\xi}_{n}: n \in \omega\}$ enumerate $\alpha$, possibly with repetitions.
For each $n \in \omega$, let ${\zeta}_{n} = \max\{{\xi}_{i}: i \leq n\}$.
Note that for each $i \leq n$, ${p}_{{\zeta}_{n}} \leq {p}_{{\xi}_{i}}$.
So it is possible to find a sequence of elements of $\omega$ $\{{k}_{0} < {k}_{1} < \dotsb \}$ such that for each $n \in \omega$, ${k}_{n} \in \dom({p}_{{\zeta}_{n}})$ and for each $i \leq n$, ${p}_{{\zeta}_{n}}({k}_{n}) \; {\subset}^{\ast} \; {p}_{{\xi}_{i}}({k}_{n})$.
Define $p \subset \omega \times \omega$ as follows.
If $m \notin \{{k}_{0} < {k}_{1} < \dotsb\}$, then $p(m) = 0$.
Suppose $m = {k}_{n}$.
Put $G({\zeta}_{n}, m, \alpha) = \{{p}_{{\zeta}_{n}}(m) \cap {x}_{\xi}(m): \xi \in \alpha \setminus F({\zeta}_{n}, m)\}$.
By (6) $\I(G({\zeta}_{n}, m, \alpha))$ is a proper ideal on ${p}_{{\zeta}_{n}}(m)$.
Since this ideal is countably generated, it is possible to find $p(m) \in {\[{p}_{{\zeta}_{n}}(m)\]}^{\omega}$ such that
\begin{enumerate}
	\item[(7)]
		for all $a \in \I(G({\zeta}_{n}, m, \alpha))$, $\lc p(m) \cap a \rc < \omega$
	\item[(8)]
		for all $a \in \I(G({\zeta}_{n}, m, \alpha))$, $\lc \left( \omega \setminus a \right) \cap \left(\omega \setminus p(m) \right)\rc = \omega$.
\end{enumerate}
Note that $p \in \P$.
Furthermore, note that if $i \in \omega$, then for any $n \geq i$, $p({k}_{n}) \subset {p}_{{\zeta}_{n}}({k}_{n}) \; {\subset}^{\ast} \; {p}_{{\xi}_{i}}({k}_{n})$.
Hence for all $\xi < \alpha$, $p \leq {p}_{\xi}$.
Next, fix $m \in \omega$ and suppose that $(p \; \cap \; {x}_{\xi})(m)$ is infinite for some $\xi < \alpha$.
Then $m = {k}_{n}$ for some (unique) $n$ and $\xi \in F({\zeta}_{n}, m)$.
However $F({\zeta}_{n}, m)$ must be a finite set.
This is because if $\xi \in F({\zeta}_{n}, m)$, then $\xi \leq {\zeta}_{n}$ and ${p}_{{\zeta}_{n}}(m) \; {\subset}^{\ast} \; {x}_{\xi}(m) \cap {x}_{{\zeta}_{n}}(m)$, and so since $m \in \dom({p}_{{\zeta}_{n}})$, $F({\zeta}_{n}, m) \subset \{{\zeta}_{n}\} \cup \{\xi < {\zeta}_{n}: \lc \left( {x}_{{\zeta}_{n}} \cap  {x}_{\xi} \right) (m) \rc = \omega\}$, which is a finite set.
So for any $m \in \omega$, $\{\xi < \alpha: \lc \left(p \cap {x}_{\xi}\right)(m)\rc = \omega\}$ is finite.
Finally, note that $p \subset {p}^{\ast}$.

Let $\U$ be $(\V, \P)$-generic with $p \in \U$.
In $\V\[\U\]$, let ${x}_{\alpha} \in \mathring{\BB}\[\U\]$ with ${x}_{\alpha} \subset p$.
In $\V$, let ${p}_{\alpha} \in \P$ be standard such that ${p}_{\alpha} \subset {x}_{\alpha} \subset p$ and ${p}_{\alpha} \forces {x}_{\alpha} \in \mathring{\BB}$.
It is clear that (3)-(5) are satisfied by $\langle {x}_{\xi}: \xi \leq \alpha \rangle$ and $\langle {p}_{\xi}: \xi \leq \alpha\rangle$.

We only need to check that (6) is satisfied.
There are several cases to consider here.
First fix $m \in \omega$ and suppose that $m \in \dom({p}_{\alpha})$.
As ${p}_{\alpha} \subset p$, $m = {k}_{n}$ for some (unique) $n \in \omega$.
Now if $\xi < \alpha$ and ${p}_{\alpha}(m) \; {\subset}^{\ast} \; {x}_{\xi}(m)$, then $p(m) \cap {x}_{\xi}(m)$ is infinite and so $\xi \in F({\zeta}_{n}, m)$.
On the other hand if $\xi \in F({\zeta}_{n}, m)$, then ${p}_{\alpha}(m) \subset p(m) \subset {p}_{{\zeta}_{n}}(m) \; {\subset}^{\ast} \; {x}_{\xi}(m)$, whence $\xi \in F(\alpha, m)$.
Therefore, $F(\alpha, m) = \{\alpha\} \cup F({\zeta}_{n}, m)$.
Put $G(\alpha, m, \alpha + 1) = \{{p}_{\alpha}(m) \cap {x}_{\xi}(m): \xi \in \left( \alpha + 1 \right) \setminus F(\alpha, m)\}$.
By (7) it is clear that $\I(G(\alpha, m, \alpha + 1))$ is the Frechet ideal on ${p}_{\alpha}(m)$.
This takes care of $\alpha$.
Next, suppose $\xi < \alpha$ and $m \in \dom({p}_{\xi})$.
Put $G(\xi, m, \alpha) = \{{p}_{\xi}(m) \cap {x}_{\zeta}(m): \zeta \in \alpha \setminus F(\xi, m)\}$ and put $G(\xi, m, \alpha + 1) = \{{p}_{\xi}(m) \cap {x}_{\zeta}(m): \zeta \in \left( \alpha + 1\right) \setminus F(\xi, m)\}$.
We know that $\I(G(\xi, m, \alpha))$ is a proper ideal on ${p}_{\xi}(m)$ and it is clear that $\I(G(\xi, m, \alpha)) = \I(G(\xi, m, \alpha + 1))$ unless ${p}_{\xi}(m) \cap {x}_{\alpha}(m) \notin \I(G(\xi, m, \alpha))$.
Suppose this is the case.
In particular, ${p}_{\xi}(m) \cap {x}_{\alpha}(m)$ is infinite.
Since ${x}_{\alpha}(m) \subset p(m)$ and ${p}_{\xi}(m) \subset {x}_{\xi}(m)$, it follows that $m = {k}_{n}$ for some (unique) $n$ and $\xi \in F({\zeta}_{n}, m)$.
Moreover, if $\xi < {\zeta}_{n}$, then since ${\zeta}_{n} \in \alpha \setminus F(\xi, m)$ and since ${x}_{\alpha}(m) \subset p(m) \subset {p}_{{\zeta}_{n}}(m) \subset {x}_{{\zeta}_{n}}(m)$, we have that ${p}_{\xi}(m) \cap {x}_{\alpha}(m) \subset {p}_{\xi}(m) \cap {x}_{{\zeta}_{n}}(m) \in \I(G(\xi, m, \alpha))$.
Therefore, $\xi = {\zeta}_{n}$.
Thus we need to show that $\I(G({\zeta}_{n}, m, \alpha + 1))$ is a proper ideal on ${p}_{{\zeta}_{n}}(m)$.
For this it suffices to show that $\omega \setminus \left( {p}_{{\zeta}_{n}}(m) \cap {x}_{\alpha}(m)\right) \notin \I(G({\zeta}_{n}, m, \alpha))$.
Note that since ${x}_{\alpha}(m) \subset p(m) \subset {p}_{{\zeta}_{n}}(m)$, ${p}_{{\zeta}_{n}}(m) \cap {x}_{\alpha}(m) = {x}_{\alpha}(m)$.
However it is clear from (8) that $\omega \setminus {x}_{\alpha}(m) \notin \I(G({\zeta}_{n}, m, \alpha))$ and we are done.
\end{proof}
Note that our argument does not rely on $\mathring{\BB}$ being closed under finite intersections.
\bibliographystyle{amsplain}
\bibliography{Bibliography}
\end{document}